\documentclass[12pt]{amsart}
\usepackage{amssymb,amscd,amsmath,amsthm,url}
\usepackage[all]{xy}
\usepackage{fullpage}
\usepackage{enumitem}
\usepackage{tikz-cd}
\usepackage{appendix}
\usepackage{xcolor}

\usepackage[colorlinks=true,
            citecolor=blue,
            linkcolor=red,
            urlcolor=blue]{hyperref}
\hypersetup{linktoc=all}

\newtheorem{Theorem}{Theorem}[section]
\newtheorem{Proposition}{Proposition}[section]
\theoremstyle{remark}
\newtheorem{Lemma}{Lemma}[section]

\newtheorem{Definition}{Definition}[section]

\numberwithin{equation}{section}

\author{Ahmadreza Khazaeipoul}
\address{Department of Mathematics, University of Toronto, 40 St. George Street, Toronto, ON, Canada M5S 2E4}
\email{Ahmadreza.Khazaeipoul@mail.utoronto.ca}

\newcommand{\Lhol}{L_{\mathrm{hol}}}

\title{An Elliptic Path Fibration and Holomorphic Morita Equivalence}

\begin{document}
\begin{abstract}
We construct an elliptic analogue of the Alekseev-Meinrenken path
fibration \cite{alekseev2009atiyah} for a simply connected complex Lie
group with a nondegenerate invariant symmetric bilinear form. A natural $2$-form on the correspondence
space induces a holomorphic symplectic groupoid over the space
of $(0,1)$-connections and a holomorphic quasi-symplectic groupoid over
the holomorphic loop group. The correspondence space defines a principal
holomorphic Morita bibundle with compatible differential forms.
\end{abstract}
\maketitle
\tableofcontents

\section{Introduction}
The path fibration of Alekseev and Meinrenken \cite{alekseev2009atiyah}
relates Hamiltonian gauge theory to quasi-Hamiltonian geometry. Its
correspondence space carries a $2$-form compatible with the actions of
the symplectic groupoid of connections on the circle and the
quasi-symplectic groupoid of conjugation on a compact Lie group.

We construct an elliptic analogue, motivated by the correspondence
between coadjoint orbits in the space of $(0,1)$-connections and
twisted-conjugacy classes in the holomorphic loop group described in
\cite[Proposition~5.13]{khesin2009geometry}. Let
\[
\Sigma=\mathbb{C}/(\mathbb{Z}+\tau\mathbb{Z})
\]
be an elliptic curve, where $\operatorname{Im}(\tau)>0$, and let $G$ be a
simply connected complex Lie group whose Lie algebra carries a nondegenerate
invariant symmetric complex-bilinear form. We introduce a correspondence
space of smooth maps on the cylinder $\mathbb C/\mathbb Z$ whose
quasi-periodicity is encoded by holomorphic loops. The space carries
commuting actions of the gauge group
\[
G_\Sigma=C^\infty(\Sigma,G)
\]
and the holomorphic loop group $\Lhol G$.

A $2$-form on this correspondence space induces a holomorphic symplectic groupoid over
\[
\mathcal A=\Omega^{0,1}(\Sigma,\mathfrak g)
\]
of $(0,1)$-connections on the trivial principal $G$-bundle. We also equip
the twisted-conjugation action groupoid
\[
\Lhol G\ltimes\Lhol G\rightrightarrows\Lhol G
\]
with a holomorphic quasi-symplectic structure.

The main result is the following.
\begin{Theorem}[Main Theorem]\label{thm:main}
Let
\[
(\mathcal G_1,\omega_1)
=\bigl(G_\Sigma\ltimes\mathcal A,\omega_1\bigr)
\rightrightarrows\mathcal A,
\qquad \mathcal A=\Omega^{0,1}(\Sigma,\mathfrak g),
\]
be the holomorphic symplectic groupoid associated with the gauge
action of
\[
G_\Sigma=C^\infty(\Sigma,G)
\]
on $\mathcal A$. Let
\[
(\mathcal G_2,\omega_2)
=\bigl(\Lhol G\ltimes\Lhol G,\omega_2\bigr)
\rightrightarrows\Lhol G
\]
be the holomorphic quasi-symplectic groupoid associated with twisted
conjugation.
Then $(\mathcal P,\varpi)$ is a principal holomorphic Morita bibundle
between these groupoids, with compatible $2$-forms and background
$3$-forms.
\end{Theorem}

Section~2 constructs the correspondence space and its $2$-form.
Section~3 establishes the groupoid structures, Section~4 proves the
action identities, and Section~5 proves the Morita equivalence. The
appendix gives the finite-dimensional identities used in
Lemma~\ref{lem:G2-quasi-symplectic}.

\subsection*{Acknowledgements}
I thank my advisor, Eckhard Meinrenken, for introducing me to this
problem, his support, and many insightful discussions.
\section{The Elliptic Path Fibration}
\subsection{Conventions and standing assumptions}\label{sec:conventions}
Throughout, $G$ denotes a simply connected complex Lie group with Lie
algebra $\mathfrak g$, and $\langle\cdot,\cdot\rangle$ denotes a fixed
nondegenerate symmetric $\operatorname{Ad}$-invariant $\mathbb C$-bilinear
form on $\mathfrak g$.

Pairings of Lie-algebra-valued forms include exterior multiplication.
In particular, the pairing of two $1$-forms is skew-symmetric. The
differential $d$ acts on the relevant parameter space, while
$\partial_{\bar z}$ differentiates in the cylinder variable.
We use the complex orientation on $\Sigma$ and its fundamental annulus.
With both boundary circles parametrized by increasing real coordinate,
\begin{equation}\label{eq:boundary-orientation}
\partial\Sigma_1=(\mathbb R/\mathbb Z)-(\mathbb R/\mathbb Z+\tau).
\end{equation}
Thus Stokes' theorem takes the form
\begin{equation}\label{eq:stokes-convention}
\int_{\Sigma_1}\frac{\partial f}{\partial\bar z}\,dz\wedge d\bar z
=\int_{\mathbb R/\mathbb Z}\bigl(f(z+\tau)-f(z)\bigr)\,dz.
\end{equation}
Fundamental vectors are defined using positive exponentials. For the
actions introduced below, this means
\[
X_{\mathcal P}(\gamma)=X\gamma,
\qquad \xi_{\mathcal P}(\gamma)=\gamma\xi.
\]

For a topological Lie algebra $\mathfrak h$, we write $\mathfrak h^*$
for its continuous dual. On $\Lhol\mathfrak g$ we use the contour pairing
\[
B(X,Y):=\int_{\mathbb R/\mathbb Z}\langle X(z),Y(z)\rangle\,dz.
\]
This pairing is nondegenerate. Contour deformation gives
\begin{equation}\label{eq:translation-pairing}
B(\kappa X,\kappa Y)=B(X,Y).
\end{equation}
The image of $B$ in the continuous dual is called the regular dual.

To embed $(0,1)$-forms in $\mathfrak g_\Sigma^*$, for $A=A_{\bar z}\,d\bar z\in\mathcal A$ and
$\xi\in\mathfrak g_\Sigma$, set
\[
\langle A,\xi\rangle_\Sigma
:=\int_\Sigma\langle A_{\bar z},\xi\rangle\,dz\wedge d\bar z.
\]
Thus $\mu:\mathcal A\to\mathfrak g_\Sigma^*$ is defined by
\[
\mu(A)(\xi)=\langle A,\xi\rangle_\Sigma,
\]
and $d\mu$ is the corresponding linear embedding. Integrals over
$\Sigma$ are computed on a fundamental domain.
Symplectic nondegeneracy is understood in the weak sense: contraction
with the $2$-form is injective. We use the relative-closedness convention
$d\omega=s^*H-t^*H$ and the IM convention
\[
\sigma(a)(v)=\omega(a,T\epsilon(v)),
\qquad a\in\ker(Ts)|_M,\quad v\in TM,
\]
where $\epsilon$ denotes the unit map.

Our bibundle convention is $d\beta=q^*H_2-p^*H_1$, with plus signs in
both action identities (Proposition~\ref{prop:action-identities}).

\subsection{The elliptic curve}
Fix $\tau\in\mathbb C$ with $\operatorname{Im}(\tau)>0$, and let
\[
\Sigma:=\mathbb C/(\mathbb Z+\tau\mathbb Z)
\]
be the corresponding elliptic curve.
Let $\Sigma_1\subset\mathbb C/\mathbb Z$ denote the image of the strip
\[
0\leq\operatorname{Im}(z)\leq\operatorname{Im}(\tau).
\]
Then $\Sigma$ is obtained from $\Sigma_1$ by identifying the boundary
circles
\[
\mathbb R/\mathbb Z\quad\text{and}\quad\mathbb R/\mathbb Z+\tau
\]
by translation by $\tau$.

\subsection{Gauge group and the holomorphic loop group}
Let
\[
G_\Sigma:=C^\infty(\Sigma,G)
\]
denote the gauge group of the trivial principal $G$-bundle over $\Sigma$,
with Lie algebra $\mathfrak g_\Sigma=C^\infty(\Sigma,\mathfrak g)$, and let
\[
\Lhol G:=\{\psi:\mathbb C/\mathbb Z\to G\mid\psi\text{ is holomorphic}\}
\]
be the holomorphic loop group, with Lie algebra
\[
\Lhol\mathfrak g
=\{X:\mathbb C/\mathbb Z\to\mathfrak g\mid X\text{ is holomorphic}\}.
\]
Via the map
\[
\mathbb C/\mathbb Z\longrightarrow\mathbb C^*,
\qquad z\longmapsto w=\exp(2\pi iz),
\]
we identify $\mathbb C/\mathbb Z$ with $\mathbb C^*$. Consequently,
$\Lhol G$ may be identified with the space of holomorphic maps
$\mathbb C^*\to G$.

Let
\[
\kappa:\Lhol G\longrightarrow\Lhol G,
\qquad (\kappa(\psi))(z)=\psi(z+\tau),
\]
denote translation by $\tau$. The \emph{twisted conjugation} action is
\[
\psi\cdot g=\kappa(\psi)g\psi^{-1}.
\]

\subsection{The correspondence space}
\begin{Definition}[Correspondence space]\label{def:P}
Define
\[
\mathcal P
=\Bigl\{\gamma\in C^\infty(\mathbb C/\mathbb Z,G)\ \Bigm|\
z\mapsto\gamma(z+\tau)\gamma(z)^{-1}\text{ is holomorphic}\Bigr\}.
\]
\end{Definition}
The space $\mathcal P$ carries commuting left and right group actions.
The right action of $G_\Sigma$ is
\[
(\gamma\cdot u)(z)=\gamma(z)u(z),
\]
while the left action of $\Lhol G$ is
\[
(\psi\cdot\gamma)(z)=\psi(z)\gamma(z).
\]
\begin{Definition}\label{def:pq}
Define
\[
q:\mathcal P\longrightarrow\Lhol G,
\qquad q(\gamma)(z)=\gamma(z+\tau)\gamma(z)^{-1},
\]
and
\[
p:\mathcal P\longrightarrow\mathcal A:=\Omega^{0,1}(\Sigma,\mathfrak g),
\qquad p(\gamma)=\gamma^{-1}\bar\partial\gamma.
\]
\end{Definition}
The map $p$ takes values in $\mathcal A$: writing
$h=q(\gamma)$, the relation $\gamma(z+\tau)=h(z)\gamma(z)$ and
$\bar\partial h=0$ give $p(\gamma)(z+\tau)=p(\gamma)(z)$.

\begin{Lemma}\label{lem:q}
The map $q:\mathcal P\to\Lhol G$ is surjective. Its fibers are precisely
the free $G_\Sigma$-orbits. Moreover, $q$ is equivariant for the left
$\Lhol G$-action and twisted conjugation on $\Lhol G$.
\end{Lemma}
\begin{proof}
Equivariance follows from
\[
q(\psi\gamma)(z)=\psi(z+\tau)q(\gamma)(z)\psi(z)^{-1}.
\]
To prove surjectivity, let $h:\mathbb C/\mathbb Z\to G$ be holomorphic.
Consider the annulus bounded by $\mathbb R/\mathbb Z$ and
$\mathbb R/\mathbb Z+\tau$. Since $G$ is simply connected, the boundary
loop $h|_{\mathbb R/\mathbb Z}$ is smoothly homotopic to the identity.
Prescribe a smooth map $\gamma$ to equal the identity on a collar of
the lower boundary and to equal $h(z-\tau)$ on a collar of the upper
boundary, and interpolate smoothly on the remaining annulus.
The relation
\[
\gamma(z+\tau)=h(z)\gamma(z)
\]
extends $\gamma$ uniquely to the cylinder. The collar prescriptions
ensure smooth agreement, including all derivatives, across each seam.
The resulting map belongs to $\mathcal P$ and satisfies $q(\gamma)=h$.

The right $G_\Sigma$-action preserves the fibers of $q$. Conversely,
if $q(\gamma_1)=q(\gamma_2)$, then
\[
u(z)=\gamma_2(z)^{-1}\gamma_1(z)
\]
satisfies $u(z+\tau)=u(z)$ and therefore descends to $G_\Sigma$.
Hence $\gamma_1=\gamma_2u$. The action is free by pointwise cancellation.
\end{proof}

\begin{Lemma}\label{lem:p}
The map $p:\mathcal P\to\Omega^{0,1}(\Sigma,\mathfrak g)$ is surjective.
Its fibers are precisely the free $\Lhol G$-orbits. It is equivariant
for the right gauge action
\[
A\cdot u=u^{-1}Au+u^{-1}\bar\partial u.
\]
\end{Lemma}
\begin{proof}
Let $A\in\Omega^{0,1}(\Sigma,\mathfrak g)$. We seek a smooth map
$\gamma:\mathbb C/\mathbb Z\to G$ satisfying
\[
\gamma^{-1}\bar\partial\gamma=A.
\]
Any such solution automatically belongs to $\mathcal P$. Indeed, put
$\psi(z)=\gamma(z+\tau)\gamma(z)^{-1}$. Then
\begin{align*}
\psi^{-1}\bar\partial\psi
&=\operatorname{Ad}_{\gamma(z)}A(z+\tau)
  -(\bar\partial\gamma(z))\gamma(z)^{-1}\\
&=\operatorname{Ad}_{\gamma(z)}\bigl(A(z+\tau)-A(z)\bigr)=0.
\end{align*}
Thus $\psi$ is holomorphic.

For existence, pull $A$ back to the cylinder
$C:=\mathbb C/\mathbb Z\cong\mathbb C^*$.
The resulting $(0,1)$-connection is integrable for dimensional reasons.
It defines a holomorphic structure on the smoothly trivial principal
$G$-bundle over the Stein curve $C$. By the Oka-Grauert \cite{grauert1958analytische} this
bundle is holomorphically trivial. Relative to the given smooth
trivialization, a holomorphic trivialization gives a smooth map $\gamma$
satisfying $\gamma^{-1}\bar\partial\gamma=A$.

Finally, if $p(\gamma_1)=p(\gamma_2)$, then
\[
\psi=\gamma_2\gamma_1^{-1}
\]
is holomorphic, so $\gamma_2=\psi\gamma_1$. Freeness follows from
pointwise cancellation. The gauge-equivariance identity is
\[
p(\gamma u)=u^{-1}p(\gamma)u+u^{-1}\bar\partial u.
\]
\end{proof}

\begin{Lemma}\label{lem:p-q-holomorphic-submersions}
The structure maps
\[
q:\mathcal P\longrightarrow\Lhol G,
\qquad p:\mathcal P\longrightarrow\mathcal A
\]
are holomorphic surjective submersions. Their differentials are given by
\eqref{eq:dq} and \eqref{eq:dp} below.
\end{Lemma}
\begin{proof}
Surjectivity follows from Lemmas~\ref{lem:q} and~\ref{lem:p}.

For $q$, write a tangent vector in right trivialization as
\[
\delta\gamma(z)=\xi(z)\gamma(z),\qquad h=q(\gamma).
\]
Differentiating gives
\begin{align*}
(T_\gamma q)(\delta\gamma)(z)
&=\xi(z+\tau)\gamma(z+\tau)\gamma(z)^{-1}
  -\gamma(z+\tau)\gamma(z)^{-1}\xi(z)\\
&=\xi(z+\tau)h(z)-h(z)\xi(z).
\end{align*}
Hence
\begin{equation}\label{eq:dq}
(T_\gamma q)(\delta\gamma)(z)h(z)^{-1}
=\xi(z+\tau)-\operatorname{Ad}_{h(z)}\xi(z).
\end{equation}
This differential is complex linear, with kernel given by the
infinitesimal right gauge directions. Local principal trivializations
imply split surjectivity of $Tq$.

For $p$, write $\delta\gamma=\gamma\xi$ and
$A=\gamma^{-1}\bar\partial\gamma$. Then
\begin{align*}
(T_\gamma p)(\gamma\xi)
&=\left.\frac{d}{dt}\right|_{t=0}
(\gamma e^{t\xi})^{-1}\bar\partial(\gamma e^{t\xi})\\
&=\bar\partial\xi+[A,\xi].
\end{align*}
Thus
\begin{equation}\label{eq:dp}
(T_\gamma p)(\gamma\xi)=D_A^{0,1}\xi,
\qquad D_A^{0,1}:=\bar\partial+[A,\,\cdot\,].
\end{equation}
In the trivialization determined by $\gamma$,
\[
D_A^{0,1}
=\operatorname{Ad}_{\gamma^{-1}}\circ\bar\partial
 \circ\operatorname{Ad}_{\gamma}.
\]
In particular, if $\delta\gamma=r\gamma$, then
\[
(T_\gamma p)(\delta\gamma)
=\operatorname{Ad}_{\gamma^{-1}}\bar\partial r.
\]
This differential is complex linear. The assumed local holomorphic
principal trivializations make $p$ and $q$ holomorphic submersions.
\end{proof}

\subsection{The $2$-form on $\mathcal P$}
The correspondence form is modelled on the Alekseev-Meinrenken path
fibration \cite{alekseev2009atiyah}.

Let
\[
\operatorname{ev}_z:\mathcal P\longrightarrow G
\]
denote evaluation at $z\in\mathbb C/\mathbb Z$, and let
$\theta^L,\theta^R\in\Omega^1(G,\mathfrak g)$ be the left and right
Maurer-Cartan forms.
\begin{Definition}\label{def:varpi}
Define the holomorphic $2$-form $\varpi\in\Omega^2(\mathcal P)$ by
\begin{equation}\label{eq:varpi}
\begin{split}
\varpi={}&-\frac12\int_{\Sigma_1}
\left\langle\operatorname{ev}_z^*\theta^R,
\frac{\partial}{\partial\bar z}
 (\operatorname{ev}_z^*\theta^R)\right\rangle\,dz\wedge d\bar z\\
&+\frac12\int_{\mathbb R/\mathbb Z}
\left\langle\operatorname{ev}_{z+\tau}^*\theta^L,
\operatorname{ev}_z^*\theta^L\right\rangle\,dz.
\end{split}
\end{equation}
\end{Definition}

\subsection{Differential and Contraction Identities}
Define the holomorphic $3$-form
\[
\eta=\int_{\mathbb R/\mathbb Z}\operatorname{ev}_z^*\eta_G\,dz
\quad\text{on }\Lhol G,
\]
where
\[
\eta_G=\frac1{12}\langle[\theta^L,\theta^L],\theta^L\rangle
=\frac1{12}\langle[\theta^R,\theta^R],\theta^R\rangle
\]
is the Cartan $3$-form on $G$.
For $X\in\Lhol\mathfrak g$, define
\[
\alpha(X)=\frac12\int_{\mathbb R/\mathbb Z}
\left(\langle X(z),\operatorname{ev}_z^*\theta^L\rangle
+\langle X(z+\tau),\operatorname{ev}_z^*\theta^R\rangle\right)\,dz.
\]

\begin{Proposition}[Structure Identities for $\varpi$]
\label{prop:varpi-identities}
The $2$-form $\varpi$ satisfies the following identities.
\begin{enumerate}
\item $d\varpi=q^*\eta$.
\item For $X\in\Lhol\mathfrak g$,
\[
\iota_{X_{\mathcal P}}\varpi=-q^*\alpha(X).
\]
\item For $\xi\in\mathfrak g_\Sigma$,
\[
\iota_{\xi_{\mathcal P}}\varpi=-p^*\langle d\mu,\xi\rangle,
\]
where $\mu(A)(\xi)=\langle A,\xi\rangle_\Sigma$.
\end{enumerate}
\end{Proposition}
\begin{proof}
Write $\varpi=\varpi_1+\varpi_2$, where
\[
\varpi_1=-\frac12\int_{\Sigma_1}
\left\langle\operatorname{ev}_z^*\theta^R,
\frac{\partial}{\partial\bar z}(\operatorname{ev}_z^*\theta^R)
\right\rangle\,dz\wedge d\bar z,
\]
and
\[
\varpi_2=\frac12\int_{\mathbb R/\mathbb Z}
\left\langle\operatorname{ev}_{z+\tau}^*\theta^L,
\operatorname{ev}_z^*\theta^L\right\rangle\,dz.
\]

\smallskip
\noindent\textbf{(i) The differential $d\varpi=q^*\eta$.}
For $\gamma\in\mathcal P$,
\[
q(\gamma)(z)=\gamma(z+\tau)\gamma(z)^{-1}.
\]
Let $F:G\times G\to G$ be given by $F(a,b)=ab^{-1}$.
The Polyakov-Wiegmann identity gives
\[
F^*\eta_G=\operatorname{pr}_1^*\eta_G-
\operatorname{pr}_2^*\eta_G+
\frac12d\left\langle\operatorname{pr}_1^*\theta^L,
\operatorname{pr}_2^*\theta^L\right\rangle.
\]
Since $\operatorname{ev}_z\circ q=F\circ
(\operatorname{ev}_{z+\tau},\operatorname{ev}_z)$, we obtain
\begin{equation}\label{eq:q-pullback-eta}
\begin{split}
q^*\operatorname{ev}_z^*\eta_G
={}&\operatorname{ev}_{z+\tau}^*\eta_G-
\operatorname{ev}_z^*\eta_G\\
&+\frac12d\left\langle\operatorname{ev}_{z+\tau}^*\theta^L,
\operatorname{ev}_z^*\theta^L\right\rangle.
\end{split}
\end{equation}

We next compute $d\varpi_1$. The right Maurer-Cartan equation is
\[
d\theta^R=\frac12[\theta^R,\theta^R].
\]
Differentiating under the integral sign gives
\begin{align*}
d\varpi_1
&=-\frac14\int_{\Sigma_1}
\left\langle\operatorname{ev}_z^*[\theta^R,\theta^R],
\frac{\partial}{\partial\bar z}\operatorname{ev}_z^*\theta^R
\right\rangle\,dz\wedge d\bar z\\
&\qquad+\frac14\int_{\Sigma_1}
\left\langle\operatorname{ev}_z^*\theta^R,
\frac{\partial}{\partial\bar z}
\operatorname{ev}_z^*[\theta^R,\theta^R]
\right\rangle\,dz\wedge d\bar z.
\end{align*}
Put
\[
C(z)=\operatorname{ev}_z^*
\langle\theta^R,[\theta^R,\theta^R]\rangle.
\]
By invariance and the graded symmetry of the pairing,
\begin{align*}
\frac{\partial C}{\partial\bar z}
&=\left\langle\frac{\partial}{\partial\bar z}
\operatorname{ev}_z^*\theta^R,
\operatorname{ev}_z^*[\theta^R,\theta^R]\right\rangle
+\left\langle\operatorname{ev}_z^*\theta^R,
\frac{\partial}{\partial\bar z}
\operatorname{ev}_z^*[\theta^R,\theta^R]\right\rangle\\
&=3\left\langle\operatorname{ev}_z^*[\theta^R,\theta^R],
\frac{\partial}{\partial\bar z}\operatorname{ev}_z^*\theta^R
\right\rangle.
\end{align*}
Consequently,
\begin{align*}
d\varpi_1
&=-\frac12\int_{\Sigma_1}
\left\langle\operatorname{ev}_z^*[\theta^R,\theta^R],
\frac{\partial}{\partial\bar z}\operatorname{ev}_z^*\theta^R
\right\rangle\,dz\wedge d\bar z
+\frac14\int_{\Sigma_1}\frac{\partial C}{\partial\bar z}
\,dz\wedge d\bar z\\
&=\left(-\frac16+\frac14\right)
\int_{\Sigma_1}\frac{\partial C}{\partial\bar z}
\,dz\wedge d\bar z\\
&=\frac1{12}\int_{\mathbb R/\mathbb Z}
\bigl(C(z+\tau)-C(z)\bigr)\,dz.
\end{align*}
Here we used \eqref{eq:stokes-convention}. Thus
\begin{equation}\label{eq:d-varpi-1}
d\varpi_1=-\int_{\mathbb R/\mathbb Z}
\left(\operatorname{ev}_z^*\eta_G-
\operatorname{ev}_{z+\tau}^*\eta_G\right)\,dz.
\end{equation}
On the other hand,
\[
d\varpi_2=\frac12\int_{\mathbb R/\mathbb Z}
d\left\langle\operatorname{ev}_{z+\tau}^*\theta^L,
\operatorname{ev}_z^*\theta^L\right\rangle\,dz.
\]
Combining this identity with \eqref{eq:d-varpi-1} gives
\begin{align*}
d\varpi
&=\int_{\mathbb R/\mathbb Z}
\Bigl(\operatorname{ev}_{z+\tau}^*\eta_G-
\operatorname{ev}_z^*\eta_G
+\frac12d\langle\operatorname{ev}_{z+\tau}^*\theta^L,
\operatorname{ev}_z^*\theta^L\rangle\Bigr)\,dz\\
&=\int_{\mathbb R/\mathbb Z}
q^*\operatorname{ev}_z^*\eta_G\,dz
=q^*\eta,
\end{align*}
by \eqref{eq:q-pullback-eta}.

\smallskip
\noindent\textbf{(ii) Contraction with the $\Lhol G$-action.}
Let $X\in\Lhol\mathfrak g$. The fundamental vector for left
multiplication is
\[
X_{\mathcal P}(\gamma)(z)=X(z)\gamma(z).
\]
In particular,
\[
\iota_{X_{\mathcal P}}\operatorname{ev}_z^*\theta^R=X(z).
\]
Therefore
\begin{align*}
\iota_{X_{\mathcal P}}\varpi_1
&=-\frac12\int_{\Sigma_1}
\left\langle X(z),\frac{\partial}{\partial\bar z}
\operatorname{ev}_z^*\theta^R\right\rangle\,dz\wedge d\bar z\\
&\qquad+\frac12\int_{\Sigma_1}
\left\langle\operatorname{ev}_z^*\theta^R,
\frac{\partial X(z)}{\partial\bar z}\right\rangle\,dz\wedge d\bar z.
\end{align*}
Since $X$ is holomorphic, the second term vanishes. If
\[
f(z)=\langle X(z),\operatorname{ev}_z^*\theta^R\rangle,
\]
then
\[
\frac{\partial f}{\partial\bar z}
=\left\langle X(z),\frac{\partial}{\partial\bar z}
\operatorname{ev}_z^*\theta^R\right\rangle.
\]
Using \eqref{eq:stokes-convention}, we obtain
\begin{align}
\iota_{X_{\mathcal P}}\varpi_1
&=\frac12\int_{\mathbb R/\mathbb Z}
\langle X(z),\operatorname{ev}_z^*\theta^R\rangle\,dz\notag\\
&\qquad-\frac12\int_{\mathbb R/\mathbb Z}
\langle X(z+\tau),\operatorname{ev}_{z+\tau}^*\theta^R\rangle\,dz.
\label{eq:contraction-varpi1}
\end{align}

For the boundary term, contraction gives
\begin{align*}
\iota_{X_{\mathcal P}}\varpi_2
&=-\frac12\int_{\mathbb R/\mathbb Z}
\left\langle\operatorname{Ad}_{\gamma(z)^{-1}}X(z),
\operatorname{ev}_{z+\tau}^*\theta^L\right\rangle\,dz\\
&\qquad+\frac12\int_{\mathbb R/\mathbb Z}
\left\langle\operatorname{Ad}_{\gamma(z+\tau)^{-1}}X(z+\tau),
\operatorname{ev}_z^*\theta^L\right\rangle\,dz.
\end{align*}
By invariance of the bilinear form, this becomes
\begin{align}
\iota_{X_{\mathcal P}}\varpi_2
&=-\frac12\int_{\mathbb R/\mathbb Z}
\left\langle X(z),\operatorname{Ad}_{\gamma(z)}
\operatorname{ev}_{z+\tau}^*\theta^L\right\rangle\,dz\notag\\
&\qquad+\frac12\int_{\mathbb R/\mathbb Z}
\left\langle X(z+\tau),\operatorname{Ad}_{\gamma(z+\tau)}
\operatorname{ev}_z^*\theta^L\right\rangle\,dz.
\label{eq:contraction-varpi2}
\end{align}
We now use
\[
q^*\operatorname{ev}_z^*\theta^L
=\operatorname{Ad}_{\gamma(z)}\operatorname{ev}_{z+\tau}^*\theta^L
-\operatorname{ev}_z^*\theta^R
\]
and
\[
q^*\operatorname{ev}_z^*\theta^R
=\operatorname{ev}_{z+\tau}^*\theta^R
-\operatorname{Ad}_{\gamma(z+\tau)}\operatorname{ev}_z^*\theta^L.
\]
Adding \eqref{eq:contraction-varpi1} and
\eqref{eq:contraction-varpi2}, we find
\begin{align*}
\iota_{X_{\mathcal P}}\varpi
&=-\frac12\int_{\mathbb R/\mathbb Z}
\left\langle X(z),
\operatorname{Ad}_{\gamma(z)}\operatorname{ev}_{z+\tau}^*\theta^L
-\operatorname{ev}_z^*\theta^R\right\rangle\,dz\\
&\qquad-\frac12\int_{\mathbb R/\mathbb Z}
\left\langle X(z+\tau),
\operatorname{ev}_{z+\tau}^*\theta^R
-\operatorname{Ad}_{\gamma(z+\tau)}\operatorname{ev}_z^*\theta^L
\right\rangle\,dz\\
&=-\frac12\int_{\mathbb R/\mathbb Z}
\left(\langle X(z),q^*\operatorname{ev}_z^*\theta^L\rangle
+\langle X(z+\tau),q^*\operatorname{ev}_z^*\theta^R\rangle\right)\,dz\\
&=-q^*\alpha(X).
\end{align*}

\smallskip
\noindent\textbf{(iii) Contraction with the $G_\Sigma$-action.}
Let $\xi\in\mathfrak g_\Sigma$. First observe that
\[
d\left(\gamma(z)^{-1}\frac{\partial\gamma(z)}{\partial\bar z}\right)
=\operatorname{Ad}_{\gamma(z)^{-1}}
\frac{\partial}{\partial\bar z}(\operatorname{ev}_z^*\theta^R).
\]
Consequently,
\[
p^*\langle d\mu,\xi\rangle
=\int_{\Sigma_1}
\left\langle\operatorname{Ad}_{\gamma(z)}\xi(z),
\frac{\partial}{\partial\bar z}(\operatorname{ev}_z^*\theta^R)
\right\rangle\,dz\wedge d\bar z.
\]
For the right action $\gamma\cdot u=\gamma u$, the fundamental vector is
$\xi_{\mathcal P}(\gamma)=\gamma\xi$, and hence
\[
\iota_{\xi_{\mathcal P}}\operatorname{ev}_z^*\theta^R
=\operatorname{Ad}_{\gamma(z)}\xi(z).
\]
It follows that
\begin{align*}
\iota_{\xi_{\mathcal P}}\varpi_1
&=-\frac12\int_{\Sigma_1}
\left\langle\operatorname{Ad}_{\gamma(z)}\xi(z),
\frac{\partial}{\partial\bar z}\operatorname{ev}_z^*\theta^R
\right\rangle\,dz\wedge d\bar z\\
&\qquad+\frac12\int_{\Sigma_1}
\left\langle\operatorname{ev}_z^*\theta^R,
\frac{\partial}{\partial\bar z}
(\operatorname{Ad}_{\gamma(z)}\xi(z))\right\rangle\,dz\wedge d\bar z\\
&=-p^*\langle d\mu,\xi\rangle
+\frac12\int_{\Sigma_1}\frac{\partial}{\partial\bar z}
\left\langle\operatorname{ev}_z^*\theta^R,
\operatorname{Ad}_{\gamma(z)}\xi(z)\right\rangle\,dz\wedge d\bar z.
\end{align*}
Using
\[
\left\langle\operatorname{ev}_z^*\theta^R,
\operatorname{Ad}_{\gamma(z)}\xi(z)\right\rangle
=\langle\operatorname{ev}_z^*\theta^L,\xi(z)\rangle,
\]
Stokes' theorem gives
\begin{align*}
\iota_{\xi_{\mathcal P}}\varpi_1
&=-p^*\langle d\mu,\xi\rangle
+\frac12\int_{\mathbb R/\mathbb Z}
\langle\xi(z+\tau),\operatorname{ev}_{z+\tau}^*\theta^L\rangle\,dz\\
&\qquad-\frac12\int_{\mathbb R/\mathbb Z}
\langle\xi(z),\operatorname{ev}_z^*\theta^L\rangle\,dz.
\end{align*}
On the other hand,
\begin{align*}
\iota_{\xi_{\mathcal P}}\varpi_2
&=\frac12\int_{\mathbb R/\mathbb Z}
\langle\xi(z+\tau),\operatorname{ev}_z^*\theta^L\rangle\,dz\\
&\qquad-\frac12\int_{\mathbb R/\mathbb Z}
\langle\xi(z),\operatorname{ev}_{z+\tau}^*\theta^L\rangle\,dz.
\end{align*}
Since $\xi(z+\tau)=\xi(z)$, these boundary contributions cancel
pairwise. Therefore
\[
\iota_{\xi_{\mathcal P}}\varpi=-p^*\langle d\mu,\xi\rangle.
\]
\end{proof}

\section{The Symplectic and Quasi-Symplectic Groupoids}
\subsection{The symplectic groupoid $\mathcal G_1$}\label{sec:G1}
Consider the fiber product
\[
\mathcal P\times_q\mathcal P
:=\{(p_1,p_2)\in\mathcal P\times\mathcal P\mid q(p_1)=q(p_2)\}.
\]
The diagonal left action of $\Lhol G$ is
\[
\psi\cdot(p_1,p_2)=(\psi\cdot p_1,\psi\cdot p_2).
\]
Since $q$ is equivariant for twisted conjugation, this action preserves
the fiber product. Define
\[
\mathcal G_1:=(\mathcal P\times_q\mathcal P)/\Lhol G
\rightrightarrows\mathcal A,
\]
with source and target
\[
s([p_1,p_2])=p(p_2),\qquad t([p_1,p_2])=p(p_1).
\]
The quotient is a holomorphic groupoid. It is naturally isomorphic to
\[
G_\Sigma\ltimes\mathcal A\rightrightarrows\mathcal A
\]
for the left gauge action
\[
u\cdot A=uAu^{-1}-(\bar\partial u)u^{-1}.
\]
Explicitly, if $p_1=p_2u^{-1}$, the quotient arrow corresponds to
$(u,p(p_2))$. This left action is the inverse of the right gauge action
used in Lemma~\ref{lem:p}.

\begin{Lemma}
The $2$-form
\[
\operatorname{pr}_1^*\varpi-\operatorname{pr}_2^*\varpi
\]
is closed and $\Lhol G$-basic. It therefore descends to a closed
holomorphic $2$-form $\omega_1$ on $\mathcal G_1$.
\end{Lemma}
\begin{proof}
First, $\varpi$ is invariant under each fixed holomorphic loop $\psi$.
Indeed, under $\gamma\mapsto\psi\gamma$, the right Maurer-Cartan form
on the mapping space transforms by $\operatorname{Ad}_\psi$; because
$\bar\partial\psi=0$, its $\bar z$-derivative transforms in the same
way. The bulk term is therefore invariant. The left Maurer-Cartan
form is unchanged under fixed left multiplication, so the boundary
term is invariant as well.

For $X\in\Lhol\mathfrak g$, contraction with the diagonal fundamental
vector gives
\begin{align*}
&\iota_{(X_{\mathcal P},X_{\mathcal P})}
(\operatorname{pr}_1^*\varpi-\operatorname{pr}_2^*\varpi)\\
&\qquad=\operatorname{pr}_1^*(\iota_{X_{\mathcal P}}\varpi)
-\operatorname{pr}_2^*(\iota_{X_{\mathcal P}}\varpi)\\
&\qquad=-\operatorname{pr}_1^*q^*\alpha(X)
+\operatorname{pr}_2^*q^*\alpha(X)=0,
\end{align*}
since $q\circ\operatorname{pr}_1=q\circ\operatorname{pr}_2$ on the
fiber product. Thus the form is basic.
Finally,
\begin{align*}
d(\operatorname{pr}_1^*\varpi-\operatorname{pr}_2^*\varpi)
&=\operatorname{pr}_1^*q^*\eta-
\operatorname{pr}_2^*q^*\eta=0.
\end{align*}
The local principal quotient therefore carries the descended closed
holomorphic form $\omega_1$.
\end{proof}

\begin{Lemma}
The form $\omega_1$ is multiplicative.
\end{Lemma}
\begin{proof}
On the fiberwise pair groupoid
$\mathcal P\times_q\mathcal P\rightrightarrows\mathcal P$,
\[
(x_1,x_2)\circ(x_2,x_3)=(x_1,x_3).
\]
On triples with common $q$-value,
\begin{align*}
m^*(\operatorname{pr}_1^*\varpi-\operatorname{pr}_2^*\varpi)
&=\operatorname{pr}_{13}^*
(\operatorname{pr}_1^*\varpi-\operatorname{pr}_2^*\varpi)\\
&=\operatorname{pr}_{12}^*
(\operatorname{pr}_1^*\varpi-\operatorname{pr}_2^*\varpi)
+\operatorname{pr}_{23}^*
(\operatorname{pr}_1^*\varpi-\operatorname{pr}_2^*\varpi).
\end{align*}
A composable pair of quotient arrows can locally be represented by a
triple of this form: use the unique holomorphic loop aligning the two
middle representatives in their common $p$-fiber. The local principal
trivializations make this alignment holomorphic. Thus the displayed
identity descends and proves multiplicativity on $\mathcal G_1$.
\end{proof}

\begin{Lemma}\label{lem:kernel}
The intersection
\[
\ker(\varpi)\cap\ker(Tq)
\]
is trivial at every point of $\mathcal P$.
\end{Lemma}
\begin{proof}
Fix $\gamma\in\mathcal P$ and let
$V\in\ker\varpi_\gamma\cap\ker T_\gamma q$.
By the principal-fiber description, there is a unique
$\xi\in\mathfrak g_\Sigma$ such that
\[
V=\xi_{\mathcal P}(\gamma)=\gamma\xi.
\]
The contraction identity gives, for every $W\in T_\gamma\mathcal P$,
\[
0=\varpi_\gamma(V,W)
=-\left\langle T_\gamma p(W),\xi\right\rangle_\Sigma.
\]
Since $T_\gamma p$ is surjective,
\[
\langle\dot A,\xi\rangle_\Sigma=0
\quad\text{for every }\dot A\in T_{p(\gamma)}\mathcal A.
\]
Nondegeneracy of the integral pairing implies $\xi=0$, and hence $V=0$.
\end{proof}

Let
\[
\rho_{\omega_1}:A(\mathcal G_1)\longrightarrow T^*\mathcal A
\]
denote the IM map:
\[
\langle\rho_{\omega_1}(a),v\rangle
=\omega_1(a,T\epsilon(v)).
\]

\begin{Proposition}\label{prop:G1-symp}
The IM map $\rho_{\omega_1}$ is injective. The form $\omega_1$ is nondegenerate, so
\[
(\mathcal G_1,\omega_1)\rightrightarrows\mathcal A
\]
is a holomorphic symplectic groupoid.
\end{Proposition}
\begin{proof}
The Lie algebroid has the quotient description
\[
A(\mathcal G_1)=T_q\mathcal P/\Lhol G.
\]
At a unit represented by $(\gamma,\gamma)$, an algebroid vector can be
represented by $(\xi_{\mathcal P}(\gamma),0)$. A unit tangent is
represented by $(W,W)$. If $v=T_\gamma p(W)$, then
\begin{align*}
\langle\rho_{\omega_1}([\xi_{\mathcal P}]),v\rangle
&=(\operatorname{pr}_1^*\varpi-\operatorname{pr}_2^*\varpi)
((\xi_{\mathcal P},0),(W,W))\\
&=\varpi(\xi_{\mathcal P},W)
=-\langle v,\xi\rangle_\Sigma.
\end{align*}
Consequently,
\[
\rho_{\omega_1}([\xi_{\mathcal P}])=-\langle d\mu,\xi\rangle.
\]
If this covector is zero, nondegeneracy of the pairing gives $\xi=0$.
Equivalently, its pullback is $\iota_{\xi_{\mathcal P}}\varpi=0$, and
Lemma~\ref{lem:kernel} applies. Thus the IM map is injective.
For the left gauge-action parameter $u=e^{t\zeta}$, the representative
has $\xi=-\zeta$; this accounts for the sign when using that algebroid
trivialization.

To prove weak nondegeneracy on the whole groupoid, let a vector in
$\ker\omega_1$ be represented by
\[
(V_1,V_2)\in T_{(\gamma_1,\gamma_2)}(\mathcal P\times_q\mathcal P).
\]
Testing against $(\xi_{\mathcal P}(\gamma_1),0)$ and
$(0,\xi_{\mathcal P}(\gamma_2))$, respectively, gives
\[
T_{\gamma_1}p(V_1)=0,\qquad T_{\gamma_2}p(V_2)=0.
\]
Hence there are $X_1,X_2\in\Lhol\mathfrak g$ with
\[
V_1=X_1\gamma_1,\qquad V_2=X_2\gamma_2.
\]
Put $h=q(\gamma_1)=q(\gamma_2)$ and $X=X_1-X_2$.
The fiber-product tangent condition implies
\[
X(z+\tau)h(z)-h(z)X(z)=0.
\]
Furthermore, for any tangent pair $(W_1,W_2)$ with common $Tq$-image $w$,
\begin{align*}
0&=\varpi(V_1,W_1)-\varpi(V_2,W_2)\\
 &=-\alpha_h(X_1)(w)+\alpha_h(X_2)(w)
 =-\alpha_h(X)(w).
\end{align*}
Every $w\in T_h\Lhol G$ occurs, since $q$ is a submersion. Thus
$\alpha_h(X)=0$. The anchor equation gives
$\operatorname{Ad}_{h(z)^{-1}}X(z+\tau)=X(z)$, so for any
$Y=\theta_h^L(w)$,
\[
\alpha_h(X)(w)=\int_{\mathbb R/\mathbb Z}\langle X(z),Y(z)\rangle\,dz.
\]
The contour pairing is nondegenerate; hence $X=0$ and $X_1=X_2$.
Therefore $(V_1,V_2)$ is tangent to the diagonal $\Lhol G$-orbit and
represents zero in the quotient. This proves weak nondegeneracy.
Closedness, holomorphicity, and multiplicativity were established above.
\end{proof}

\subsection{The quasi-symplectic groupoid $\mathcal G_2$}\label{sec:G2}
Consider the action groupoid
\[
\mathcal G_2:=\Lhol G\ltimes\Lhol G\rightrightarrows\Lhol G
\]
with
\[
s(\psi,g)=g,\qquad t(\psi,g)=\kappa(\psi)g\psi^{-1}.
\]
Let $\operatorname{pr}_1,\operatorname{pr}_2:\mathcal G_2\to\Lhol G$
be the projections onto the action and base factors, respectively.
Define
\begin{equation}\label{eq:omega2-def}
\omega_2=-\frac12\int_{\mathbb R/\mathbb Z}
\left(
\left\langle\operatorname{Ad}_{g(z)^{-1}}\Lambda_{z+\tau},\Lambda_z\right\rangle
-\left\langle\Lambda_{z+\tau},\Theta_z^R\right\rangle
-\left\langle\Lambda_z,\Theta_z^L\right\rangle
\right)\,dz,
\end{equation}
where
\[
\Lambda_z:=(\operatorname{ev}_z\circ\operatorname{pr}_1)^*\theta^L,
\qquad
\Theta_z^L:=(\operatorname{ev}_z\circ\operatorname{pr}_2)^*\theta^L,
\qquad
\Theta_z^R:=(\operatorname{ev}_z\circ\operatorname{pr}_2)^*\theta^R.
\]

\begin{Lemma}\label{lem:G2-quasi-symplectic}
The groupoid
\[
(\mathcal G_2=\Lhol G\ltimes\Lhol G,\omega_2)
\rightrightarrows(\Lhol G,\eta)
\]
is holomorphic quasi-symplectic. More precisely:
\begin{enumerate}
\item $\omega_2$ is multiplicative:
\[
m^*\omega_2=\operatorname{pr}_1^*\omega_2+
\operatorname{pr}_2^*\omega_2;
\]
\item $\omega_2$ is relatively closed:
\[
d\omega_2=s^*\eta-t^*\eta;
\]
\item its IM map is $\sigma_2=\alpha$, and
$\ker\rho_g\cap\ker\sigma_{2,g}=\{0\}$;
\item the image of $(\rho_g,\sigma_{2,g})$ is maximally isotropic
in the regular tangent-cotangent pairing.
\end{enumerate}
\end{Lemma}
\begin{proof}
Write
\[
\omega_2=\int_{\mathbb R/\mathbb Z}\Omega_z\,dz,
\]
where
\[
\Omega_z=-\frac12\left(
\left\langle\operatorname{Ad}_{g(z)^{-1}}\Lambda_{z+\tau},\Lambda_z\right\rangle
-\left\langle\Lambda_{z+\tau},\Theta_z^R\right\rangle
-\left\langle\Lambda_z,\Theta_z^L\right\rangle
\right).
\]
For fixed $z$, put
\[
a=\psi(z),\qquad b=\psi(z+\tau),\qquad h=g(z).
\]
Then $\Omega_z$ is the pullback of
\[
\Omega=-\frac12\left(
\langle\operatorname{Ad}_{h^{-1}}\phi_2,\phi_1\rangle
-\langle\phi_2,\Theta^R\rangle
-\langle\phi_1,\Theta^L\rangle
\right),
\]
where
\[
\phi_1=a^{-1}da,\qquad\phi_2=b^{-1}db,
\qquad\Theta^L=h^{-1}dh,\qquad\Theta^R=dh\,h^{-1}.
\]
The corresponding finite-dimensional action is
\[
(a,b)\cdot h=bha^{-1},\qquad
s(a,b,h)=h,\qquad t(a,b,h)=bha^{-1}.
\]
For this two-sided action, the finite-dimensional identities contain
additional terms. Appendix~\ref{app:cartan-dirac} gives
\begin{equation}\label{eq:finite-relative}
d\Omega=s^*\eta_G-t^*\eta_G+\operatorname{pr}_b^*\eta_G
-\operatorname{pr}_a^*\eta_G,
\end{equation}
as well as the multiplicativity defect. The additional terms cancel
only after imposing the translation relation and integrating.

\smallskip
\noindent\emph{Relative closedness.}
Pulling back \eqref{eq:finite-relative} gives
\begin{align*}
d\Omega_z
={}&(\operatorname{ev}_z\circ s)^*\eta_G
-(\operatorname{ev}_z\circ t)^*\eta_G\\
&+(\operatorname{ev}_{z+\tau}\circ\operatorname{pr}_1)^*\eta_G
-(\operatorname{ev}_z\circ\operatorname{pr}_1)^*\eta_G.
\end{align*}
The last two terms have equal contour integrals: their coefficients
are holomorphic and $1$-periodic in $z$, so the contours can be translated
by $\tau$. Therefore
\begin{align*}
d\omega_2
&=\int_{\mathbb R/\mathbb Z}d\Omega_z\,dz\\
&=s^*\int_{\mathbb R/\mathbb Z}\operatorname{ev}_z^*\eta_G\,dz
-t^*\int_{\mathbb R/\mathbb Z}\operatorname{ev}_z^*\eta_G\,dz\\
&=s^*\eta-t^*\eta.
\end{align*}

\smallskip
\noindent\emph{Multiplicativity.}
A composable pair is of the form
\[
(\psi_1,\kappa(\psi_2)g\psi_2^{-1}),\qquad(\psi_2,g),
\]
and its product is $(\psi_1\psi_2,g)$. Indeed,
\begin{align*}
\psi_1\cdot(\psi_2\cdot g)
&=\kappa(\psi_1)\bigl(\kappa(\psi_2)g\psi_2^{-1}\bigr)\psi_1^{-1}\\
&=\kappa(\psi_1\psi_2)g(\psi_1\psi_2)^{-1}.
\end{align*}
Let $\lambda_{i,z}^{L,R}$ denote the Maurer-Cartan forms of the
$i$th action variable evaluated at $z$. The finite-dimensional defect
computed in the appendix gives
\begin{align*}
&m^*\Omega_z-\operatorname{pr}_1^*\Omega_z-
\operatorname{pr}_2^*\Omega_z\\
&\qquad=-\frac12\left(
\langle\lambda_{1,z+\tau}^L,\lambda_{2,z+\tau}^R\rangle
-\langle\lambda_{1,z}^L,\lambda_{2,z}^R\rangle\right).
\end{align*}
Both coefficient forms are holomorphic in $z$, and one is the
$\tau$-translate of the other. Their contour integrals agree. Hence
\[
m^*\omega_2=\operatorname{pr}_1^*\omega_2+
\operatorname{pr}_2^*\omega_2.
\]

\smallskip
\noindent\emph{Minimal degeneracy.}
The Lie algebroid is the action algebroid
\[
A(\mathcal G_2)=\Lhol G\times\Lhol\mathfrak g,
\]
with anchor
\[
\rho_g(X)(z)=X(z+\tau)g(z)-g(z)X(z).
\]
After left trivialization,
\[
\theta_g^L(\rho_g(X))(z)
=\operatorname{Ad}_{g(z)^{-1}}X(z+\tau)-X(z).
\]
Evaluating \eqref{eq:omega2-def} at the unit $(1,g)$ on
$((X,0),(0,v))$ gives
\begin{align*}
\sigma_{2,g}(X)(v)
&=\frac12\int_{\mathbb R/\mathbb Z}
\left(\langle X(z),\theta_g^L(v)(z)\rangle
+\langle X(z+\tau),\theta_g^R(v)(z)\rangle\right)\,dz\\
&=\alpha_g(X)(v).
\end{align*}
Suppose $\rho_g(X)=0$ and $\sigma_{2,g}(X)=0$. The anchor equation implies
\[
\operatorname{Ad}_{g(z)^{-1}}X(z+\tau)=X(z).
\]
For $v\in T_g\Lhol G$, put $Y(z)=\theta_g^L(v)(z)$. Then
$\theta_g^R(v)(z)=\operatorname{Ad}_{g(z)}Y(z)$, and
\begin{align*}
\alpha_g(X)(v)
&=\frac12\int_{\mathbb R/\mathbb Z}
\left(\langle X(z),Y(z)\rangle
+\langle X(z+\tau),\operatorname{Ad}_{g(z)}Y(z)\rangle\right)\,dz\\
&=\frac12\int_{\mathbb R/\mathbb Z}
\left(\langle X(z),Y(z)\rangle
+\langle\operatorname{Ad}_{g(z)^{-1}}X(z+\tau),Y(z)\rangle\right)\,dz\\
&=\int_{\mathbb R/\mathbb Z}\langle X(z),Y(z)\rangle\,dz.
\end{align*}
Thus $B(X,Y)=0$ for every holomorphic loop $Y$, and nondegeneracy of
$B$ implies $X=0$. This proves
\[
\ker\rho_g\cap\ker\sigma_{2,g}=\{0\}.
\]

For maximal isotropy, put
\[
T_gX=\operatorname{Ad}_{g^{-1}}\kappa(X).
\]
This operator preserves $B$ by \eqref{eq:translation-pairing} and
pointwise $\operatorname{Ad}$-invariance. In left trivialization,
\[
\rho_g=T_g-I,\qquad\sigma_{2,g}=\tfrac12(I+T_g),
\]
where the second equality uses the regular-dual identification.
The image is isotropic because
\begin{align*}
&B((T_g-I)X,\tfrac12(I+T_g)Y)
+B((T_g-I)Y,\tfrac12(I+T_g)X)=0.
\end{align*}
If $(v,\beta)$ is orthogonal to the image, nondegeneracy of $B$ gives
\[
(T_g^{-1}-I)\beta+\tfrac12(I+T_g^{-1})v=0.
\]
Taking $X=\beta-v/2$ yields $T_gX=\beta+v/2$, and therefore
\[
v=(T_g-I)X,\qquad\beta=\tfrac12(I+T_g)X.
\]
Thus the image equals its orthogonal complement within the regular
tangent-cotangent space.

\end{proof}
The pair $(\mathcal G_2,\omega_2)$ is a holomorphic analogue of the
Alekseev-Malkin-Meinrenken construction \cite{alekseev1998lie}, with
background Cartan $3$-form $\eta$. Its infinitesimal geometry is the
twisted Cartan-Dirac geometry interpreted through the regular pairing.

\section{Hamiltonian Actions on the Correspondence Space}
Consider the left action of $\mathcal G_1$ and the right action of
$\mathcal G_2$ on $\mathcal P$. The left action is induced by the
fiberwise pair-groupoid action:
\[
[(p_1,p_2)]\cdot p_2=p_1.
\]
More generally, if $p(\gamma)=p(p_2)$, let $\psi\in\Lhol G$ be the unique
loop with $\gamma=\psi p_2$, and set
\[
[(p_1,p_2)]\cdot\gamma=\psi p_1.
\]
This definition does not depend on the representative of the quotient
arrow. The right action is
\[
\gamma\cdot(\psi,g)=\psi^{-1}\gamma,
\qquad q(\gamma)=t(\psi,g).
\]
The output has $q$-value $s(\psi,g)=g$.

\begin{Proposition}\label{prop:action-identities}
For the groupoid
\[
\mathcal G_2=\Lhol G\ltimes\Lhol G\rightrightarrows\Lhol G,
\qquad \psi\cdot g=\kappa(\psi)g\psi^{-1},
\]
the left and right actions on $\mathcal P$ satisfy
\[
a_L^*\varpi=\operatorname{pr}_{\mathcal G_1}^*\omega_1+
\operatorname{pr}_{\mathcal P}^*\varpi
\]
and
\[
a_R^*\varpi=\operatorname{pr}_{\mathcal P}^*\varpi+
\operatorname{pr}_{\mathcal G_2}^*\omega_2.
\]
\end{Proposition}
\begin{proof}
\emph{Left action identity.}
On the fiberwise pair groupoid, the source and target are
\[
s(p_1,p_2)=p_2,\qquad t(p_1,p_2)=p_1,
\]
and the action is $a_L((p_1,p_2),p_2)=p_1$. Thus, on this lifted domain,
\begin{align*}
a_L^*\varpi
&=\operatorname{pr}_1^*\varpi\\
&=(\operatorname{pr}_1^*\varpi-\operatorname{pr}_2^*\varpi)
+\operatorname{pr}_2^*\varpi\\
&=\operatorname{pr}_{\mathcal G_1}^*\omega_1+
\operatorname{pr}_{\mathcal P}^*\varpi.
\end{align*}
Descent to the full action domain follows from the loop-action identity
below and the equality $q(p_1)=q(p_2)$.

\medskip
\emph{Right action identity.}
Let
\[
a_R:\mathcal P\times_{q,t}\mathcal G_2\longrightarrow\mathcal P,
\qquad a_R(\gamma,(\psi,g))=\psi^{-1}\gamma.
\]
First consider the left loop action
\[
\ell:\Lhol G\times\mathcal P\longrightarrow\mathcal P,
\qquad \ell(\psi,\delta)=\psi\delta.
\]
Set
\[
F(\psi,\delta)=(\psi,q(\delta))\in\mathcal G_2.
\]
Then
\[
q\circ\ell=t\circ F,
\]
since
\begin{align*}
q(\psi\delta)(z)
&=\psi(z+\tau)\delta(z+\tau)
\bigl(\psi(z)\delta(z)\bigr)^{-1}\\
&=\kappa(\psi)(z)q(\delta)(z)\psi(z)^{-1}.
\end{align*}
We will show
\begin{equation}\label{eq:left-loop-action}
\ell^*\varpi=\operatorname{pr}_{\mathcal P}^*\varpi-F^*\omega_2.
\end{equation}
The right-action identity will then follow by writing
$\gamma=\psi\delta$.

Write $\varpi=\varpi_{\mathrm{bulk}}+\varpi_{\partial}$, where
\[
\varpi_{\mathrm{bulk}}=-\frac12\int_{\Sigma_1}
\left\langle\operatorname{ev}_z^*\theta^R,
\frac{\partial}{\partial\bar z}(\operatorname{ev}_z^*\theta^R)
\right\rangle\,dz\wedge d\bar z
\]
and
\[
\varpi_{\partial}=\frac12\int_{\mathbb R/\mathbb Z}
\left\langle\operatorname{ev}_{z+\tau}^*\theta^L,
\operatorname{ev}_z^*\theta^L\right\rangle\,dz.
\]
On $\Lhol G\times\mathcal P$, define
\[
\Gamma_z(\psi,\delta)=\delta(z),\qquad
\Psi_z(\psi,\delta)=\psi(z).
\]
Then $\operatorname{ev}_z\circ\ell=m\circ(\Psi_z,\Gamma_z)$, where
$m:G\times G\to G$ is multiplication.

\medskip
\emph{The bulk term.}
The right Maurer-Cartan form satisfies
\[
m^*\theta^R=\operatorname{pr}_1^*\theta^R+
\operatorname{Ad}_{\operatorname{pr}_1}\operatorname{pr}_2^*\theta^R.
\]
Consequently,
\[
(\operatorname{ev}_z\circ\ell)^*\theta^R
=\Psi_z^*\theta^R+\operatorname{Ad}_{\Psi_z}\Gamma_z^*\theta^R.
\]
Since $\psi$ is holomorphic,
\[
\frac{\partial}{\partial\bar z}(\Psi_z^*\theta^R)=0,
\qquad
\frac{\partial}{\partial\bar z}(\Psi_z^*\theta^L)=0.
\]
Hence
\[
\frac{\partial}{\partial\bar z}
((\operatorname{ev}_z\circ\ell)^*\theta^R)
=\operatorname{Ad}_{\Psi_z}\frac{\partial}{\partial\bar z}
(\Gamma_z^*\theta^R).
\]
Using invariance and
$\operatorname{Ad}_{\Psi_z^{-1}}\Psi_z^*\theta^R=\Psi_z^*\theta^L$,
we obtain
\begin{align*}
&\left\langle(\operatorname{ev}_z\circ\ell)^*\theta^R,
\frac{\partial}{\partial\bar z}
((\operatorname{ev}_z\circ\ell)^*\theta^R)\right\rangle\\
&\qquad=\left\langle\Gamma_z^*\theta^R,
\frac{\partial}{\partial\bar z}(\Gamma_z^*\theta^R)\right\rangle
+\frac{\partial}{\partial\bar z}
\left\langle\Psi_z^*\theta^L,\Gamma_z^*\theta^R\right\rangle.
\end{align*}
Stokes' theorem in the form \eqref{eq:stokes-convention} gives
\begin{align}
\ell^*\varpi_{\mathrm{bulk}}
&=\operatorname{pr}_{\mathcal P}^*\varpi_{\mathrm{bulk}}
-\frac12\int_{\mathbb R/\mathbb Z}
\left\langle\Psi_{z+\tau}^*\theta^L,
\Gamma_{z+\tau}^*\theta^R\right\rangle\,dz\notag\\
&\qquad+\frac12\int_{\mathbb R/\mathbb Z}
\left\langle\Psi_z^*\theta^L,\Gamma_z^*\theta^R\right\rangle\,dz.
\label{eq:bulk-action-correction}
\end{align}

\medskip
\emph{The boundary term.}
Set
\[
Q_z=\Gamma_{z+\tau}\Gamma_z^{-1},\qquad
\lambda_z=\Psi_z^*\theta^L,\qquad
\alpha_z=\Gamma_z^*\theta^L,\qquad
\rho_z=\Gamma_z^*\theta^R.
\]
The left Maurer-Cartan multiplication formula gives
\[
(\operatorname{ev}_z\circ\ell)^*\theta^L
=\operatorname{Ad}_{\Gamma_z^{-1}}\lambda_z+\alpha_z.
\]
Substitution into the boundary term yields
\begin{align*}
\ell^*\varpi_{\partial}
&=\operatorname{pr}_{\mathcal P}^*\varpi_{\partial}
+\frac12\int_{\mathbb R/\mathbb Z}
\left\langle\operatorname{Ad}_{\Gamma_{z+\tau}^{-1}}\lambda_{z+\tau},
\operatorname{Ad}_{\Gamma_z^{-1}}\lambda_z\right\rangle\,dz\\
&\qquad+\frac12\int_{\mathbb R/\mathbb Z}
\left\langle\operatorname{Ad}_{\Gamma_{z+\tau}^{-1}}\lambda_{z+\tau},
\alpha_z\right\rangle\,dz\\
&\qquad+\frac12\int_{\mathbb R/\mathbb Z}
\left\langle\alpha_{z+\tau},
\operatorname{Ad}_{\Gamma_z^{-1}}\lambda_z\right\rangle\,dz.
\end{align*}
The Maurer-Cartan forms of $Q_z$ satisfy
\[
Q_z^*\theta^L=\operatorname{Ad}_{\Gamma_z}
(\alpha_{z+\tau}-\alpha_z),
\qquad
Q_z^*\theta^R=\rho_{z+\tau}-\operatorname{Ad}_{Q_z}\rho_z.
\]
By invariance,
\[
\left\langle\operatorname{Ad}_{\Gamma_{z+\tau}^{-1}}\lambda_{z+\tau},
\operatorname{Ad}_{\Gamma_z^{-1}}\lambda_z\right\rangle
=\left\langle\operatorname{Ad}_{Q_z^{-1}}\lambda_{z+\tau},
\lambda_z\right\rangle.
\]
The two mixed terms can be rewritten as
\begin{align*}
\left\langle\operatorname{Ad}_{\Gamma_{z+\tau}^{-1}}\lambda_{z+\tau},
\alpha_z\right\rangle
&=\left\langle\lambda_{z+\tau},\operatorname{Ad}_{Q_z}\rho_z\right\rangle\\
&=\left\langle\lambda_{z+\tau},\rho_{z+\tau}-Q_z^*\theta^R\right\rangle
\end{align*}
and, using graded skew-symmetry,
\begin{align*}
\left\langle\alpha_{z+\tau},
\operatorname{Ad}_{\Gamma_z^{-1}}\lambda_z\right\rangle
&=-\left\langle\lambda_z,
\operatorname{Ad}_{\Gamma_z}\alpha_{z+\tau}\right\rangle\\
&=-\left\langle\lambda_z,Q_z^*\theta^L+\rho_z\right\rangle.
\end{align*}
Therefore
\begin{align}
\ell^*\varpi_{\partial}
&=\operatorname{pr}_{\mathcal P}^*\varpi_{\partial}
+\frac12\int_{\mathbb R/\mathbb Z}
\Bigl(\left\langle\operatorname{Ad}_{Q_z^{-1}}\lambda_{z+\tau},\lambda_z\right\rangle
-\left\langle\lambda_{z+\tau},Q_z^*\theta^R\right\rangle\notag\\
&\hspace{5.4cm}-\left\langle\lambda_z,Q_z^*\theta^L\right\rangle\Bigr)\,dz\notag\\
&\qquad+\frac12\int_{\mathbb R/\mathbb Z}
\left(\langle\lambda_{z+\tau},\rho_{z+\tau}\rangle
-\langle\lambda_z,\rho_z\rangle\right)\,dz.
\label{eq:boundary-action-correction}
\end{align}
The last integral cancels the correction in
\eqref{eq:bulk-action-correction}. To identify the remaining integral,
recall that
\begin{equation}\label{eq:omega2-recall}
\omega_2=-\frac12\int_{\mathbb R/\mathbb Z}
\Bigl(\left\langle\operatorname{Ad}_{g(z)^{-1}}\Lambda_{z+\tau},\Lambda_z\right\rangle
-\left\langle\Lambda_{z+\tau},\Theta_z^R\right\rangle
-\left\langle\Lambda_z,\Theta_z^L\right\rangle\Bigr)\,dz.
\end{equation}
Under $F(\psi,\delta)=(\psi,q(\delta))$, the forms $\Lambda_z$ pull
back to $\lambda_z$, and $\Theta_z^{L,R}$ pull back to
$Q_z^*\theta^{L,R}$. Thus the three-term integral, including its coefficient $\tfrac12$
in \eqref{eq:boundary-action-correction}, is $-F^*\omega_2$.
Combining the bulk and boundary contributions gives
\begin{align*}
\ell^*\varpi
&=\ell^*\varpi_{\mathrm{bulk}}+\ell^*\varpi_{\partial}\\
&=\operatorname{pr}_{\mathcal P}^*
(\varpi_{\mathrm{bulk}}+\varpi_{\partial})-F^*\omega_2\\
&=\operatorname{pr}_{\mathcal P}^*\varpi-F^*\omega_2,
\end{align*}
which proves \eqref{eq:left-loop-action}.

The map
\[
(\psi,\delta)\longmapsto(\psi\delta,(\psi,q(\delta)))
\]
identifies $\Lhol G\times\mathcal P$ with
$\mathcal P\times_{q,t}\mathcal G_2$. Under this identification,
$a_R$ is the projection to $\delta$, while the $\mathcal P$-projection
is $\ell$. Rearranging \eqref{eq:left-loop-action} gives
\[
a_R^*\varpi=\operatorname{pr}_{\mathcal P}^*\varpi+
\operatorname{pr}_{\mathcal G_2}^*\omega_2.
\]

For descent of the left-action identity, use the surjective submersion
\[
(p_1,p_2,\psi)\longmapsto([p_1,p_2],\psi p_2),
\qquad q(p_1)=q(p_2),
\]
onto the left-action domain. The output is $\psi p_1$.
Apply \eqref{eq:left-loop-action} to $p_1$ and $p_2$ and subtract.
Since their $q$-values agree, the two $\omega_2$ corrections cancel.
The remaining difference is
$\operatorname{pr}_1^*\varpi-\operatorname{pr}_2^*\varpi$, the pullback
of $\omega_1$. This proves the left-action identity on the full domain.
\end{proof}

\section{Holomorphic Morita equivalence}
A principal holomorphic bibundle $(P,p,q)$ between groupoids
$(\mathcal H_i,\omega_i)$ over $(M_i,H_i)$ carries compatible forms if
\[
d\beta=q^*H_2-p^*H_1,
\]
the left and right actions satisfy the identities in
Proposition~\ref{prop:action-identities}, and
\[
\ker\beta\cap\ker Tp\cap\ker Tq=\{0\}.
\]
Here $d\omega_i=s_i^*H_i-t_i^*H_i$, with nondegeneracy understood as
in Section~\ref{sec:conventions}.

\begin{Theorem}\label{thm:morita}
The
space $(\mathcal P,\varpi)$ is a holomorphic Morita bibundle with
compatible forms. It relates the groupoids
$(\mathcal G_1,\omega_1)$ and $(\mathcal G_2,\omega_2)$, over
$(\mathcal A,0)$ and $(\Lhol G,\eta)$, respectively:
\[
\begin{tikzcd}
{(\mathcal G_1,\omega_1)} \arrow[d,shift right] \arrow[d,shift left]
& {(\mathcal P,\varpi)} \arrow[ld,"p"'] \arrow[rd,"q"]
& {(\mathcal G_2,\omega_2)} \arrow[d,shift right] \arrow[d,shift left] \\
\mathcal A && \Lhol G.
\end{tikzcd}
\]
\end{Theorem}
\begin{proof}
The left $\mathcal G_1$-action and right $\mathcal G_2$-action commute.
Indeed, if a quotient arrow is represented by $(p_1,p_2)$ and a loop
$\psi$ acts through the right-action formula, the corresponding pair
is $(\psi^{-1}p_1,\psi^{-1}p_2)$, which represents the same quotient
arrow. All structure maps and actions are holomorphic under the
standing analytic assumptions.

For principality, suppose
$p(\gamma)=p(\delta)$, there is a unique $\psi\in\Lhol G$ such that
$\gamma=\psi\delta$. The unique arrow taking $\gamma$ to $\delta$
under the right action is
\[
(\psi,q(\delta)),
\]
whose target is $q(\gamma)$ and whose source is $q(\delta)$.
Thus the right $\mathcal G_2$-action is free and transitive on the
fibers of $p$. Its division map is holomorphic by the assumed local
principal trivializations.

Similarly, if $q(\gamma)=q(\delta)$, the unique quotient arrow taking
$\gamma$ to $\delta$ under the left action is
\[
[\delta,\gamma].
\]
The associated division map is holomorphic, and the left action is
principal with quotient $q$. Both $p$ and $q$ are holomorphic
surjective submersions.

By Proposition~\ref{prop:action-identities},
\[
a_L^*\varpi=\operatorname{pr}_{\mathcal G_1}^*\omega_1+
\operatorname{pr}_{\mathcal P}^*\varpi,
\qquad
 a_R^*\varpi=\operatorname{pr}_{\mathcal P}^*\varpi+
\operatorname{pr}_{\mathcal G_2}^*\omega_2.
\]
The background-form condition follows from
Proposition~\ref{prop:varpi-identities}:
\[
d\varpi=q^*\eta=q^*\eta-p^*0.
\]
The minimal-degeneracy condition is also satisfied, since
\[
\ker\varpi\cap\ker Tp\cap\ker Tq
\subseteq\ker\varpi\cap\ker Tq=\{0\}
\]
by Lemma~\ref{lem:kernel}.

Proposition~\ref{prop:G1-symp} and
Lemma~\ref{lem:G2-quasi-symplectic} establish the required groupoid
structures, completing the proof.
\end{proof}

\appendix
\section{Finite-dimensional Cartan-Dirac identity}\label{app:cartan-dirac}
The finite-dimensional theory of multiplicative forms and induced
Dirac structures is developed in
\cite[Section~4]{bursztyn2004integration}. We record the correction
terms for the two-sided action used in Section~\ref{sec:G2}.

On $G\times G\times G$, write
\[
\phi_1=a^{-1}da,\qquad\phi_2=b^{-1}db,
\qquad\Theta^L=h^{-1}dh,\qquad\Theta^R=dh\,h^{-1},
\]
and set
\[
\Omega=-\frac12\left(
\langle\operatorname{Ad}_{h^{-1}}\phi_2,\phi_1\rangle
-\langle\phi_2,\Theta^R\rangle
-\langle\phi_1,\Theta^L\rangle\right).
\]
For $t(a,b,h)=bha^{-1}$ and $s(a,b,h)=h$, we claim
\begin{equation}\label{eq:appendix-relative}
d\Omega=s^*\eta_G-t^*\eta_G+\operatorname{pr}_b^*\eta_G
-\operatorname{pr}_a^*\eta_G.
\end{equation}
Indeed, the multiplication and division identities for the Cartan
form give
\begin{align*}
(bh)^*\eta_G
&=b^*\eta_G+h^*\eta_G
-\frac12d\langle b^*\theta^L,h^*\theta^R\rangle,\\
(bha^{-1})^*\eta_G
&=(bh)^*\eta_G-a^*\eta_G
+\frac12d\langle(bh)^*\theta^L,a^*\theta^L\rangle.
\end{align*}
Since
\[
(bh)^*\theta^L=\operatorname{Ad}_{h^{-1}}\phi_2+\Theta^L,
\]
substitution yields
\begin{align*}
t^*\eta_G
&=b^*\eta_G+h^*\eta_G-a^*\eta_G\\
&\qquad+\frac12d\left(
-\langle\phi_2,\Theta^R\rangle
+\langle\operatorname{Ad}_{h^{-1}}\phi_2,\phi_1\rangle
+\langle\Theta^L,\phi_1\rangle\right)\\
&=b^*\eta_G+h^*\eta_G-a^*\eta_G-d\Omega.
\end{align*}
This proves \eqref{eq:appendix-relative}.

For multiplicativity, a composable pair of two-sided arrows is
\[
(a_1,b_1,b_2ha_2^{-1}),\qquad(a_2,b_2,h),
\]
with product $(a_1a_2,b_1b_2,h)$. Put
\[
u_i=a_i^*\theta^L,\qquad v_i=b_i^*\theta^L,
\qquad \widehat u_i=a_i^*\theta^R,\qquad
\widehat v_i=b_i^*\theta^R.
\]
The Maurer-Cartan multiplication formulas are
\[
(a_1a_2)^*\theta^L=\operatorname{Ad}_{a_2^{-1}}u_1+u_2,
\qquad
(b_1b_2)^*\theta^L=\operatorname{Ad}_{b_2^{-1}}v_1+v_2.
\]
For $h'=b_2ha_2^{-1}$, the base forms are
\begin{align*}
(h')^*\theta^L
&=\operatorname{Ad}_{a_2}
\bigl(\operatorname{Ad}_{h^{-1}}v_2+\Theta^L-u_2\bigr),\\
(h')^*\theta^R
&=\operatorname{Ad}_{b_2}
\bigl(v_2+\Theta^R-\operatorname{Ad}_h u_2\bigr).
\end{align*}
Substituting these expressions into the three terms of $\Omega$, and
using invariance of the pairing to cancel the terms containing
$\Theta^{L,R}$ and the mixed terms containing $\operatorname{Ad}_h$,
gives
\begin{equation}\label{eq:appendix-defect}
m^*\Omega-\operatorname{pr}_1^*\Omega-\operatorname{pr}_2^*\Omega
=-\frac12\left(\langle v_1,\widehat v_2\rangle
-\langle u_1,\widehat u_2\rangle\right).
\end{equation}
These extra terms generally do not vanish pointwise.

Under $a_i=\psi_i(z)$ and $b_i=\psi_i(z+\tau)$, the two terms on the
right of \eqref{eq:appendix-defect} are translates in $z$ and have
identical contour integrals. The same is true of the last two terms
in \eqref{eq:appendix-relative}. Integration commutes with $d$ and
pullback, so the resulting loop-group form satisfies
\[
d\omega_2=s^*\eta-t^*\eta,
\qquad m^*\omega_2=\operatorname{pr}_1^*\omega_2+
\operatorname{pr}_2^*\omega_2,
\]
as used in Lemma~\ref{lem:G2-quasi-symplectic}.

\nocite{*}
\bibliographystyle{plain}
\bibliography{Sources}
\end{document}